\documentclass[10pt]{amsart}
\usepackage{amsmath,amssymb,amsthm,setspace, times, fancyhdr, bussproofs, multicol, rotating, mathpazo}

\newtheorem{thm}{\sc Theorem}
\newtheorem{pro}[thm]{ \sc Proposition}
\newtheorem{lem}[thm]{\sc Lemma}
\theoremstyle{definition}
\newtheorem{defn}{\sc Definition}
\begin{document}

\title{ON ELIMINATION OF QUANTIFIERS IN SOME NON-CLASSICAL MATHEMATICAL THEORIES}

\author{Guillermo Badia,\; Andrew Tedder}
\maketitle

\begin{abstract}
Elimination of quantifiers is  shown to fail dramatically for a group of well-known mathematical theories (classically enjoying the property) against a wide range of relevant logical backgrounds. Furthermore, it is suggested that only by moving to more extensional underlying logics can we get the property back.

\bigskip


\bigskip

\emph{Math subject classification:} Primary, 03C10,  	03B47 ; Secondary,  	03B53. 
\end{abstract}

\section{Introduction}

The property of \emph{quantifier elimination} for a theory has the following classical uses (\cite{Hodges93}, pp. 68-69): (1) classification of structures up to elementary equivalence, (2) completeness and decidability  proofs for theories, (3) description of elementary embeddings, and (4) description of definable relations in a structure.  In this paper, we chose to focus on the latter and, as in \cite{Marker02}, take it to be the fundamental import of quantifier elimination. Our aim  is to explore elimination of quantifiers in the non-classical setting of quantificational logics which are sound with respect to the Routley-Meyer semantic framework from relevant logic \cite{Routley82}.\footnote{Quantifier elimination has been recently studied  in \cite{Ellison08} for the also non-classical context of intuitionistic logic.} Such quantificational logics include many relevant and paraconsistent logics as well as, say, classical logic.

We should alert the reader that the concern of the present pages will be with \emph{mathematical theories} (theories describing some mathematical structure), so we will disregard things such as the pure theory of equality given that we will assume equality to be a logical notion with a fixed interpretation.\footnote{Indeed, as can be extracted from \S \ref{sec:qe}, if $=$ is taken as a non-logical symbol then the pure theory of equality will certainly fail to have QE.} Some examples of mathematical theories are the theory of \emph{dense linear orderings without endpoints}, the theory of \emph{real closed fields}, the theory of \emph{ordered divisible  abelian groups} and \emph{Presburger arithmetic}. All these theories are known to have quantifier elimination in the classical setting \cite{Marker02}. Here we investigate the fate of this property in the Routley-Meyer semantic framework. 

We also look at the minimal logical principles that can  be added to our formal systems to recover quantifier elimination. This is in line with some recent investigations into the non-classical models of mathematical theories.\footnote{Papers with results which are directly relevant to our work are \cite{Restall09, Mortensen88} and \cite{Ferguson12}. Other papers in relevant arithmetic are \cite{Meyer76}, \cite{MeyerMortensen84}, and \cite{MeyerFriedman92}. For the broader class of paraconsistent arithmetics, see \cite{Priest97}, \cite{Priest00}, \cite{ParisPathmanathan06}, and \cite{ParisSirokofskich08}.}  Our toy theory will be the theory of dense linear orderings without endpoints (henceforth, DLO), which is the simplest around.\footnote{Details on the \emph{classical} model theoretic properties of DLO, and other mathematical theories considered here can be found in \cite{Marker02}, or any standard model theory textbook. Our axioms for RCF  in \S 3.2 are found in \cite{Sacks10}.} 

This paper is (fundamentally) a contribution  to the study of paraconsistent mathematical theories (i.e., theories where from a sentence and its negation not everything follows). Our work is motivated by the need of understanding just how different the behaviour of these theories is from that of their classical counterparts. We would like to stress that, surprisingly,  most of the  research done in this field with logics having a detachable conditional has used algebraic models. In that sense, \cite{Restall09} seems to be the first (and likely only) place where Routley-Meyer relational models have been employed in the study of substructural mathematical theories as we will do below.

In \S \ref{sec:log}, we review the details of the Routley-Meyer semantics as well as introduce the fundamental definitions we will be working with. In \S \ref{math}, we present the theories we will be studying. In \S \ref{sec:fqe}, we show that quantifier elimination fails for well-known mathematical theories when the underlying logic is any of the most famous systems of relevant logic. In \S \ref{sec:qe}, we discuss how to get quantifier elimination back. Finally, in \S \ref{sec:con}, we summarize our results.

\section{Logical Preliminaries}\label{sec:log}

All the  languages we will be considering will have connectives $\neg,\land,\lor,\rightarrow,\forall,\exists,\bot,\top$ (with the usual arities) and a denumerable list of individual variables $x_0,x_1,\dots$. As usual, by the \emph{signature} of a language we will mean its relation, function and constant symbols. 

We will be presenting relevant logics semantically rather than syntactically. 
The reason for this is  that Fine \cite{Fine89} has shown that the natural proof-theoretic formulation of the well-know system \textbf{R} of relevant logic plus first order quantifiers is incomplete with respect to the Routley-Meyer semantic framework.\footnote{A semantics which brings completeness back is given in \cite{Mares06}.} Hence, if one desires to find a complete axiomatization for the systems we will call \textbf{B}, \textbf{R} and \textbf{RM} they would need to go beyond what is usually understand proof-theoretically by these calculi. It is not difficult to show, however, that our semantically defined systems are indeed recursively axiomatizable using Craig's method.

\subsection*{Models}

We employ the standard ternary relation semantics, also routinely called the ``Routley-Meyer'' semantics. A \textbf{B}-model is a structure $M=\langle W,R,D,^*,s,V\rangle$ where $W$ is a non-empty set (of worlds), $R\subseteq W^3$, $D$ is a set of objects (the domain), $^*: W \longrightarrow W$, $s \in W$, $V$ is a valuation such that if $w \in W$ and $P^n$ is an $n$-ary predicate, $V(w, P^n) \subseteq D^n$, and moreover:

\begin{itemize}

\item For any $\alpha\in W$, $Rs\alpha\alpha$
\item for any $\alpha,\beta,\gamma,\delta\in W$, $Rs\alpha\beta$ and $R\beta\gamma\delta$ implies that $R\alpha\gamma\delta$
\item For any $\alpha\in W$, $\alpha^{**}=\alpha$
\item For any $\alpha,\beta\in W$, if $Rs\alpha\beta$ then $Rs\beta^*\alpha^*$
\item If $R0\alpha\beta$, then for any $P^n$, $V(\alpha, P^n) \subseteq V(\beta, P^n)$.

\end{itemize}

The special predicate = will always be interpreted as the metatheoretic classical identity unless  stated otherwise.

 The additional ternary relation restrictions necessary for the logic \textbf{R}-mingle, or \textbf{RM}, are as follows (the axioms enforced by the principles are included):\footnote{Sans-serif uppercase letters give the Combinatory names of implication axioms, and are used throughout, in line with standard usage in the relevant literature. A good introduction to Combinatory logic can be found in \cite{Bimbo11}, and details regarding the correspondence between combinators and implicational formulae can be found in \cite{AndersonBelnapDunn92}[\S71] or in \cite{DunnRestall02}.}  

\begin{itemize}
\item For any $\alpha,\beta, \gamma \in W$, if $R\alpha\beta \gamma$ then $R\alpha \gamma^*\beta^*$  \hfill$(A\rightarrow B)\rightarrow (\neg B \rightarrow \neg A)$
\item For any $\alpha,\beta,\gamma,\delta\in W$, if $R\alpha\beta\gamma$ then $R\alpha\gamma\beta$. \hfill \textsf{CI}
\item For any $\alpha\in W$, $R\alpha\alpha\alpha$. \hfill\textsf{WI}
\item If there is an $x\in W$ s.t. $R\alpha\beta x$ and $Rx\gamma\delta$ then there is a $y\in W$ s.t. $R\alpha y\delta$ and $R\beta\gamma y$. \hfill \textsf{B}
\item If $R\alpha\beta\gamma$ then either $Rs\alpha\gamma$ or $Rs\beta\gamma$ \hfill $A\rightarrow(A\rightarrow A)$
\end{itemize} 

The correspondence between combinatory logic and propositional logic support the following pairs of combinators and valid formulae:

\begin{itemize}
\item[\textsf{B}] \hfill$(A\rightarrow B)\rightarrow ((C\rightarrow A)\rightarrow (C\rightarrow B))$
\item[\textsf{WI}] \hfill $((A\rightarrow B)\land A)\rightarrow B$
\item[\textsf{CI}] \hfill $A\rightarrow ((A\rightarrow B)\rightarrow B)$
\end{itemize}

Given a model $M$, we will use $\overline{a}$ to denote a sequence of individuals of the domain of $M$. Truth of complex formulae at a world is defined as follows:

\begin{itemize}
\item $M,w\vDash \bot$ never 
\item $M,w\vDash A\land B[\overline{a}]$ iff $M,w\vDash A[\overline{a}]$ and $M,w\vDash B[\overline{a}]$
\item $M,w\vDash (A\lor B)[\overline{a}]$ iff $M,w\vDash A[\overline{a}]$ or $M,w\vDash B[\overline{a}]$
\item $M,w\vDash\neg A [\overline{a}]$ iff $M,w^*\nvDash A[\overline{a}]$
\item $M,w\vDash (A\rightarrow B)[\overline{a}]$ iff if $Rww_1w_2$ and $M,w_1\vDash A[\overline{a}]$ then $M,w_2\vDash B[\overline{a}]$
\item $M,w\vDash\forall xA[\overline{a}]$ iff for any $y\in D$, $M,w\vDash A[y/x, \overline{a}]$
\item $M,w\vDash\exists xA[\overline{a}]$ iff for some $y\in D$, $M,w\vDash A[y/x, \overline{a}]$
\end{itemize}

\begin{defn}
 For any logic \textbf{L} discussed here, set of formulas $\Gamma$ and formula $A$, we will say that $A$\emph{ is deducible from} $\Gamma$ in {\bf L} or, in symbols,  $\Gamma \vDash_{\bf L} A$ when for every {\bf L}-model $M$, we have that $M,s\vDash \Gamma$ only if $M,s \vDash A$. So deducibility only cares about what happens at the distinguished world  $s$. In particular if $\Gamma = \emptyset$ we get a definition of theoremhood in {\bf L}. In what follows, we will drop the subscript in $\vDash_{\bf L}$ when the context makes clear about what system we are talking.

\end{defn}

\begin{defn}
A theory $T$ \emph{admits quantifier elimination} if, for any formula $A$, there is a quantifier free formula $B$ such that $T\vDash A\leftrightarrow B$.
\end{defn}

Next we make a couple of important observations which will be used (often without explicit mention) throughout the paper. 

\begin{pro} \emph{(Hereditary condition)} Let M be a \textbf{B}-model.  If $Rs\alpha \beta$ and $A(x)$ is any formula, $M, \alpha \vDash A[a]$ only if $M, \beta \vDash A[a]$.\end{pro}

\begin{pro} Let M be a \textbf{B}-model.  Then $M, s \vDash A \rightarrow B [a]$ iff for any $\alpha$, $M, \alpha \vDash A  [a]$ only if  $M, \alpha \vDash B [a]$. \end{pro}

\begin{defn} 
Given some equivalence relation $\equiv$ between models, where $N\equiv M$ implies that all sentences true in one of the models are also true in the other, a theory $T$ is \emph{$\omega$-categorical with respect to $\equiv$} if for any countable models $N,M\vDash T$, $N \equiv M$.  
\end{defn}

\begin{defn}
A theory $T$ is \emph{negation-complete} if for any formula $A$, either $T\vDash A$ or $T\vDash\neg A$. 
\end{defn}

\begin{defn}
A theory $T$ is \emph{complete} if for any models $M, N$ of $T$, $M$ satisfies exactly the same sentences as $N$.
\end{defn}

Quantifier elimination is a model-theoretic property in the sense that it considerably simplifies the theory of the definable sets of the models of a given theory: these amount to the things one can define using quantifier free formulas.

\section{Some mathematical theories}\label{math}

In this section, we will review the axiomatizations of a our target mathematical theories. The theories as presented are, of course, not the classical theories which go by these names, as the underlying logic differs from classical logic. However, we retain (roughly) the same axioms as the classical theories, though stated in non-classical vocabulary. We shall be focused on properties had by classical formal theories which are lost when the underlying logic is relevant, and so we shall use the usual names for the classical theories to make explicit the comparison.

\subsection*{Dense Linear Orderings}

 In a signature with a binary relation symbol $<$, DLO without endpoints is given by the following axioms:\footnote{DLO with a paraconsistent underlying logic has been studied  in \cite{Ferguson12} as formulated in the $\rightarrow$-free fragment of our language (where $A \rightarrow B$ is replaced by $\neg A \vee B$).}

\begin{enumerate}
\item[A1] $\forall x(x<x\rightarrow\bot)$
\item[A2] $\forall x,y(x=y\lor x<y\lor y<x)$
\item[A3] $\forall x,y,z((x<y\land y<z)\rightarrow x<z)$
\item[A4] $\forall x,y(x<y\rightarrow\exists z(x<z\land z<y))$
\item[A5] $\forall x\exists y(x<y)$
\item[A6] $\forall x\exists y(y<x)$
\end{enumerate}

If DLO is formulated with $\forall  x \neg (x<x)$ instead of A1 it is a simple matter to build a one element model for it, so the following result (though easy) is a necessary preliminary.
 
\begin{pro} DLO in any logic \textbf{L} extending \textbf{B} has no finite models.
\end{pro}
\begin{proof} This is easily seen essentially as in the classical case by constructing an infinitely ascending (descending) sequence (in the ordering holding at world $s$)  of distinct elements of the domain of objects of any  \textbf{B}-model of DLO. In the reasoning one needs to use the property $Rsss$ of \textbf{B}-models in an essential way but only axioms A1, A5 (or alternatively A6) and A3.
\end{proof}

We have to insist that the failure of QE for DLO (which will be established below in \S\ref{sec:fqe})  holds, \emph{a fortiori}, for DLO with $\forall x \neg (x<x)$ replacing A1. We do not have a direct proof of the failure of QE for such alternative formulation of DLO. A similar remark applies to all the theories we will examine in this paper, hence whenever $\neg A$ occurs in a standard classical axiomatisation of some theory, we shall replace this with $A\rightarrow\bot$. For the QE result when the underlying logic is strengthened, DLO in its current form looks like what is required. 

\subsection*{Real Closed Fields}

The signature of the real closed fields has a binary predicate symbol $<$,  binary function symbols $+, \times$, unary function symbols $-, ^{-1}$, and constants $0, 1$. The theory RCF has the following axioms:\footnote{For a different (essentially second order) approach to $\mathbb{R}$ in a paraconsistent setting, see \cite{WeberJordens12}.}

\begin{enumerate}
\item[A1] $\forall x(x<x\rightarrow\bot)$
\item[A2] $\forall x,y(x=y\lor x<y\lor y<x)$
\item[A3] $\forall x,y,z((x<y\land y<z)\rightarrow x<z)$
\item[A4] $\forall x,y((0<x \land 0 < y) \rightarrow 0 < xy)$
\item[A5] $\forall x,y, z(x< y \rightarrow  x+z<y+z)$
\item[A6] $\forall x \exists y (0 < x \rightarrow x = yy)$
\item[A7] $\forall x_1, \dots, x_n \exists y (y^n+x_1y^{n-1}+ \dots + x_n =0)$ for each odd $n > 0$.
\item[A8] $\forall x, y, z ((x+y)+z=x+(y+z))$
\item[A9] $\forall x (x+0=x)$
\item[A10] $\forall x (x+(-x) =0)$
\item[A11] $\forall x, y (x+y=y+x)$
\item[A12] $\forall x, y ((xy)z = x(yz))$
\item[A13] $\forall x (x1 = x)$
\item[A14] $\forall x (\neg(x = 0) \rightarrow xx^{-1} = 1)$
\item[A15] $\forall x, y(xy = yx)$
\item[A16] $\forall x, y, z (x(y+z) = (xy)+(xz))$
\item[A17]    $0 = 1 \rightarrow \bot$     
\end{enumerate}

\subsection*{Presburger arithmetic}

The signature of Presburger arithmetic has a binary predicate symbol $<$, unary predicates $P_n$ ($n > 1$), a  binary function symbol $+$,  and constants $0, 1$. The axioms are as follows:

\begin{enumerate}
\item[A1] $\forall x(x<x\rightarrow\bot)$
\item[A2] $\forall x,y(x=y\lor x<y\lor y<x)$
\item[A3] $\forall x,y,z((x<y\land y<z)\rightarrow x<z)$

\item[A4] $\forall x,y, z(x< y \rightarrow  x+z<y+z)$

\item[A5] $\forall x, y, z ((x+y)+z=x+(y+z))$
\item[A6] $\forall x (x+0=x)$

\item[A7] $\forall x, y (x+y=y+x)$

\item[A8]    $0 < 1 $     

\item[A9]  $\forall x((x=0 \lor x<0)\lor (1<x \lor x =1))$   

\item[A10] $\forall x (P_n(x) \leftrightarrow \exists y (x = \underbrace{y + \dots + y}_{n-\emph{times}})) $ for $n > 1$

\item[A11] $\forall x  (\bigvee_{i=0}^{n-1} (P_n(x + \underbrace{1 + \dots + 1}_{i-\emph{times}}) \land \bigwedge_{i \neq j} (P_n(x + \underbrace{1 + \dots + 1}_{j-\emph{times}}) \rightarrow \bot)))  $ for $n > 1$

\end{enumerate}

\subsection*{Divisible Ordered Abelian Groups}

The signature of the theory of divisible ordered  abelian groups (DOAG) has a binary predicate symbol $<$, a unary function symbol $-$, a  binary function symbol $+$,  and the constants $0$. Here is the list of axioms: 

\begin{enumerate}
\item[A1] $\forall x(x<x\rightarrow\bot)$
\item[A2] $\forall x,y(x=y\lor x<y\lor y<x)$
\item[A3] $\forall x,y,z((x<y\land y<z)\rightarrow x<z)$

\item[A4] $\forall x,y, z(x< y \rightarrow  x+z<y+z)$

\item[A5] $\forall y \exists x (\underbrace{x+x \dots x}_{n-times} = y)$ for $n > 1$.
\item[A6] $\forall x, y, z ((x+y)+z=x+(y+z))$
\item[A7] $\forall x (x+0=x)$
\item[A8] $\forall x (x+(-x) =0)$
\item[A9] $\forall x, y (x+y=y+x)$

\end{enumerate}

\subsection*{Algebraically Closed Fields}\label{sec:acf}

The signature of the algebraically closed fields has   binary function symbols $+, \times$, unary function symbols $-, ^{-1}$, and constants $0, 1$. The theory ACF has the following axioms:

\begin{enumerate}

\item[A1] $\forall x, y, z ((x+y)+z=x+(y+z))$
\item[A2] $\forall x (x+0=x)$
\item[A3] $\forall x (x+(-x) =0)$
\item[A4] $\forall x, y (x+y=y+x)$
\item[A5] $\forall x, y ((xy)z = x(yz))$
\item[A6] $\forall x (x1 = x)$
\item[A7] $\forall x (\neg(x = 0) \rightarrow xx^{-1} = 1)$
\item[A8] $\forall x, y(xy = yx)$
\item[A9] $\forall x, y, z (x(y+z) = (xy)+(xz))$
\item[A10]    $0 = 1 \rightarrow \bot$ 
\item[A11] $\forall x_1, \dots, x_n \exists y (y^n+x_1y^{n-1}+ \dots + x_n =0)$ for each $n > 0$.    
\end{enumerate}

\section{Failure of Quantifier Elimination}\label{sec:fqe}

In this section we study the failure of QE for most of the theories from  \S \ref{math}. Note that we employ the following seriality condition: $\forall x\exists y,z (Rxyz)$. This is need for the proof, but is not motivated by any considerations other than that it is so needed. We have not discovered a means of obtaining this important result without the seriality condition.

\begin{lem}\label{1}
Let \textbf{L} be any logic  extending \textbf{B} such that its frames satisfy the condition $\forall x \exists y, z (Rxyz)$ (a seriality requirement). Then if $A(x)$ is a quantifier free formula with one free variable in DLO without endpoints, either $A(x)\leftrightarrow x=x$ or $A(x)\leftrightarrow x<x$.
\end{lem}

\begin{proof}
We proceed by structural induction on $A(x)$:

\bigskip
\noindent Base: $A(x)$ is atomic; hence $A(x)$ is either $x=x$ or $x<x$.

\bigskip
\noindent Induction step: We proceed by cases on the main connective of $A(x)$, where the induction hypothesis guarantees that any component formula $B(x)$ of $A(x)$ are such that either $B(x)\leftrightarrow x=x$ or $B(x)\leftrightarrow x<x$.

\bigskip
\noindent Case 1: $A(x)$ is $B(x)\land C(x)$. By induction hypothesis, we know that any of the following four may be true of $B(x)$, $C(x)$.

\begin{itemize}
\item[(i)] $B(x)\leftrightarrow x=x$,
\item[(ii)] $B(x)\leftrightarrow x<x$,
\item[(iii)] $C(x)\leftrightarrow x=x$,
\item[(iv)] $C(x)\leftrightarrow x<x$.
\end{itemize}

So, we have four possibilities to consider:

\begin{enumerate}
\item (i) and (iii)
\item (i) and (iv)
\item (ii) and (iii)
\item (ii) and (iv)
\end{enumerate}

Suppose (i) and (iii) hold. Then $(B(x)\land C(x))\leftrightarrow (x=x\land x=x)\leftrightarrow x=x$. The result for (ii) and (iv) is similar.

Suppose (i) and (iv) hold. Then $(x=x\land x<x)\leftrightarrow (B(x)\land C(x))$. So $(B(x)\land C(x))\leftrightarrow x<x$ given that $(x=x\land x<x)\leftrightarrow x<x$ (the right to left implication follows using A1). The result for (ii) and (iii) is similar.

\bigskip
\noindent Case 2: $A(x)$ is $B(x)\lor C(x)$. By inductive hypothesis we have the same set of possibilities as above.

Suppose that (i), (iii) hold. As in Case 1, $(x=x\lor x=x)\leftrightarrow x=x$. For (ii), (iv), $(x<x\lor x<x) \leftrightarrow x<x$.

Suppose that (i), (iv) hold. That is, $(B(x)\lor C(x))\leftrightarrow (x=x\lor x<x)$. Well, $x=x\rightarrow(x=x\lor x<x)$. Obviously, $x=x\rightarrow x=x$; we need merely note that by A1 and the transitivity of $\rightarrow$, $x<x\rightarrow x=x$. Given this, we have $(x=x\rightarrow x=x)\land(x<x\rightarrow x=x)$, and we have as an instance of an axiom of \textbf{B} that $((x=x\rightarrow x<x)\land(x<x\rightarrow x=x))\rightarrow((x=x\lor x<x)\rightarrow x=x)$, therefore $(x=x\lor x<x) \rightarrow x=x$. So $(B(x)\lor C(x))\leftrightarrow x=x$.

\bigskip
\noindent Case 3: Suppose that $A(x)$ is $\neg B(x)$. By inductive hypothesis, either $B(x)\leftrightarrow x=x$ or $B(x)\leftrightarrow x<x$. Suppose that $B(x)\leftrightarrow x=x$. Then $A(x)\leftrightarrow\neg x=x$. It is easy to check that $\neg x=x\rightarrow\bot$ is a theorem (recall that the interpretation of $=$ has been fixed at every world) and $\bot\rightarrow x<x$, hence $A(x)\leftrightarrow x<x$ using A1. Suppose that $B(x)\leftrightarrow x<x$. Then $A(x)\leftrightarrow\neg x<x$. By A1, contraposition, the fixed interpretation of $=$ which makes $x=x \leftrightarrow \top$ hold, and transitivity of $\rightarrow$, it follows that $x=x\leftrightarrow\neg x<x$, hence $A(x)\leftrightarrow  x=x$. 

\bigskip
\noindent Case 4: Suppose that $A(x)$ is $B(x)\rightarrow C(x)$. Again, we have the possibilities (1)--(4) as in Case 1. 

Suppose that (i), (iii) hold. Then $(B(x)\rightarrow C(x))\leftrightarrow(x=x\rightarrow x=x)$. Since $x=x\leftrightarrow\top$, we know that $(\top\leftrightarrow\top)\leftrightarrow(x=x\leftrightarrow x=x)$. Similarly, $(\top\leftrightarrow\top)\leftrightarrow\top$, and so $(x=x\leftrightarrow x=x)\leftrightarrow x=x$, and $A(x)\leftrightarrow x=x$.

Suppose that (ii), (iv) hold. Then $(B(x)\rightarrow C(x))\leftrightarrow(x<x\rightarrow x<x)\leftrightarrow(\bot\leftrightarrow\bot)\leftrightarrow\top\leftrightarrow x=x$. So $A(x)\leftrightarrow x=x$.

Suppose that (i), (iv) hold. Then $(B(x)\rightarrow C(x))\leftrightarrow(x=x\rightarrow x<x)$. We know that $(x=x\rightarrow x<x)\leftrightarrow (\top\rightarrow\bot)$. Note that using the seriality condition on our frames,   $(\top\rightarrow\bot)\leftrightarrow\bot$ must hold. It suffices to note that never $\top\rightarrow\bot$ holds at any world $\alpha$ of any model, for otherwise, using seriality,  $\bot$ would have to hold at some world, which is impossible.  Hence,  $(x=x\rightarrow x<x)\leftrightarrow\bot$, and the definition of $\bot$ gives us that $\bot\leftrightarrow x<x$, so $A(x)\leftrightarrow x<x$.

Suppose that (ii), (iii) hold. Then $(B(x)\rightarrow C(x))\leftrightarrow(x<x\rightarrow x=x)$ and we know that $(x<x\rightarrow x=x)\rightarrow x=x$. Similarly, we know that $x=x\rightarrow(x<x\rightarrow x=x)$ given that $x<x\rightarrow x=x$ holds at every world in every model, and so $A(x)\leftrightarrow x=x$.

\end{proof}

\begin{thm}  Let \textbf{L} be any relevant logic extending \textbf{B} up to \textbf{RM}. Then QE fails for DLO in \textbf{L}.
\end{thm}
\begin{proof}
\bigskip
Let $N=\langle W,D,R,^*,s,V\rangle$ the model   where:

\begin{itemize}
\item $W=\{s,t\}$
\item $D=\mathbb{Q}$
\item $R=\{\langle s,s,s\rangle,\langle s,t,t\rangle,\langle t,s,s\rangle,\langle t,s,t\rangle,\langle t,t,s\rangle,\langle t,t,t\rangle, \langle s,t,s\rangle\}$
\item $V$ assigns formulae containing $<$ the values consonant with the usual ordering on $Q$ at world $s$. At world $t$, $V$ assigns the usual ordering on $<$ for rationals in the interval $[2,3]$ and assigns false to any atomic formula involving rationals outside this interval. At every world, $=$ is interpreted as real identity.
\item $s$ is the designated world.
\item $^*=\{\langle s,t\rangle,\langle t,s\rangle\}$
\end{itemize}
The reader can easily convince themselves that DLO is validated at $s$. On the other hand, it is a tedious, but  mechanical  process to check that $N$ satisfies the frame conditions to validate the logic \textbf{RM}.\footnote{It has been verified using PROVER9, a computer proof assistant McCune \cite{Prover9}}

\bigskip
The definition of $\{2\}$ at world $s$ is the formula
\begin{center}
 $\neg\exists x(x<y)\land(\forall x (x=x)\rightarrow\exists x(y<x))$.
\end{center}

 This is due to the fact that for any $a$, 
\begin{center}
$N,s\vDash \neg\exists x(x<y)\land(\forall x (x=x)\rightarrow\exists x(y<x))[a]$ iff $a=2$.
\end{center}
 First, it is easy to see that $N,s\vDash \neg\exists x(x<y)\land(\forall x (x=x)\rightarrow\exists x(y<x))[2]$ since $N,t\nvDash \exists x(x<y)[2]$ by definition of $N$ and  $N,s\vDash\forall x (x=x)\rightarrow\exists x(y<x)[2]$ given that  $N,w\vDash \exists x(y<x)[2]$ for any $w\in W$. Finally, note that by definition of $N$, 2 is the only element $a \in D$ for which $N,t\nvDash \exists x(x<y)[a]$ and $N,t\vDash \exists x(y<x)[a]$.

These results show that $\neg\exists x(x<y)\land(\forall x (x=x)\rightarrow\exists x(y<x))$ is not equivalent to $\top$ or $\bot$, (that is, is not equivalent to $x<x$ or $x=x$). Now, given that the  models for {\bf RM} satisfy the seriality condition, by Lemma \ref{1},  $\neg\exists x(x<y)\land(\forall x (x=x)\rightarrow\exists x(y<x))$ is not equivalent to a quantifier free formula.\footnote{The formula $\forall y, z (\neg (x<y<z) \vee x<y<z)$ is another example of something that has no quantifierless equivalent in DLO. So, in fact, $\rightarrow$ is not essential to construct such a formula.}
\end{proof}

\subsection*{Failure of $\omega$-categoricity with respect to an equivalence relation}

\begin{pro}
There exist countable models $N,M\vDash$ DLO s.t. not $N\equiv M$.
\end{pro}

\begin{proof}
Let $M=\langle W,D,R,^*,s,V\rangle$ where:
\begin{itemize}
\item $W=\{s\}$, $R=W^3$, $^*=W^2$
\item $D=\mathbb{Q}$
\item $s$ is the designated world
\item $V$ assigns the usual $<$-ordering on $\mathbb{Q}$.
\item $V$ assigns = real identity.
\end{itemize}

\bigskip
Given the definition of $V$, $M\vDash$ DLO, and since $D=\mathbb{Q}$, $M$ is countable. $M,s\vDash\forall x (x = x)\rightarrow$ A3 and $M,s\nvDash$ A3 $\land\neg$ A3.

\bigskip
Let $N=\langle W,D,R,^*,s,V\rangle$ where:
\begin{itemize}
\item $W=\{s,t\}$
\item $D=\mathbb{Q}$
\item $R=\{\langle s,s,s\rangle,\langle s,t,t\rangle\}$
\item $^*=\{\langle s,t\rangle,\langle t,s\rangle\}$
\item $s$ is the designated world
\item $V$ assigns the usual $<$-ordering on $\mathbb{Q}$ at $s$, and assigns $<$ to $\varnothing$ at $t$.
\item $V$ assigns = real identity at every world.
\end{itemize}

Given the definition of $V$ for $<$ at $s$, $N\vDash$ DLO. Again, $N$ is countable. We need only note that $N,s\nvDash\forall x (x = x)\rightarrow$ A3 and $N,s\vDash$ A3 $\land\neg$ A3. Hence, we have formulae $A,B$ s.t. $M\vDash A$ and $N\nvDash A$, and $M\nvDash B$ and $N\vDash B$, and so not $M\equiv N$. 

\end{proof}

\subsection*{Failure of Negation Completeness}

\begin{pro}
There exists a formula $A$ such that DLO in RM proves neither $A$ nor $\neg A$.
\end{pro}

\begin{proof}
Consider the RM model $M$ where:

\begin{itemize}
\item $W=\{s\}$, $R=W^3$, $^*=W^2$
\item $D=\mathbb{Q}$
\item $s$ is the designated world
\item $V$ assigns the usual $<$-ordering on $\mathbb{Q}$
\end{itemize}

On this model, the `classical' model for RM, we have that $M,s\nvDash\neg(\forall x\neg x<x\rightarrow\forall x\exists yy<x)$. In the model $N$ defined in Theorem 3, we have that $N,t\nvDash\forall x\neg x<x\rightarrow\forall x\exists yy<x$ because $N,t\vDash\forall x\neg x<x$ and $N,t\nvDash\forall x\exists yy<x$. Hence, it is the case that DLO$\nvDash \forall x\neg x<x\rightarrow\forall x\exists yy<x$ and DLO$\nvDash\neg(\forall x\neg x<x\rightarrow\forall x\exists yy<x)$.

\end{proof}

So, for DLO, we have the failure of negation completeness, quantifier elimination, $\omega$-categoricity, and it's fairly obvious that we also have the failure of model-completeness. Essentially, in any relevant logic, DLO fails to have any of the nice model-theoretic properties it enjoys in the classical context. At base, one can chalk this up to the fact that in the relevant setting, we have much more power to construct countermodels.

\begin{thm}\label{t} Let \textbf{L} be any relevant logic between \textbf{B} and \textbf{RM}. Then QE fails for RCF in \textbf{L}.
\end{thm}

\begin{proof}

Consider the first order  structure  $\langle \mathbb{R}, <_{\mathbb{R}}, +_{\mathbb{R}}, -_{\mathbb{R}}, ^{-1_{\mathbb{R}}}, \times_{\mathbb{R}}, 0_{\mathbb{R}}, 1_{\mathbb{R}} \rangle$. By the upward L\"owenheim-Skolem theorem, there is an elementary (in the sense of classical first order logic) extension   $\langle \mathbb{R}^{\prime}, <_{\mathbb{R}^{\prime}}, +_{\mathbb{R}^{\prime}}, -_{\mathbb{R}^{\prime}},  ^{-1_\mathbb{R}^{\prime}}, \times_{\mathbb{R}^{\prime}}, 0_{\mathbb{R}^{\prime}}, 1_{\mathbb{R}^{\prime}} \rangle$ of  $\langle \mathbb{R}, <_{\mathbb{R}}, +_{\mathbb{R}}, -_{\mathbb{R}}, ^{-1_{\mathbb{R}}}, \times_{\mathbb{R}}, 0_{\mathbb{R}}, 1_{\mathbb{R}} \rangle$ such that $|\mathbb{R}^{\prime}| = 2^{2^{\omega}}$. So in particular,  $\langle \mathbb{R}^{\prime}, <_{\mathbb{R}^{\prime}}, +_{\mathbb{R}^{\prime}}, -_{\mathbb{R}^{\prime}},  ^{-1_\mathbb{R}^{\prime}}, \times_{\mathbb{R}^{\prime}}, 0_{\mathbb{R}^{\prime}}, 1_{\mathbb{R}^{\prime}} \rangle$  is a model of the classical theory of real closed fields. Furthermore, it can be guaranteed that $|\mathbb{R}| < |\{x \in \mathbb{R}^{\prime} : 0_{\mathbb{R}^{\prime}} <_{\mathbb{R}^{\prime}} x\}|$.

Define the ordering $<^{\prime}$ on $\mathbb{R}^{\prime}$ as:

\begin{center}
$x<'y$ iff $0 <_{\mathbb{R}^{\prime}}y-x$ and $y-x \in \mathbb{R}$, 
\end{center}
where the expression $y-x$ abbreviates $y+(-x)$ as usual. First, we see that for $x \in \mathbb{R}^{\prime}$, it cannot be the case that  $x<'x$ since $0 <_{\mathbb{R}^{\prime}}x-x = 0$ is false. So $<'$ is irreflexive. If $x, y, z \in  \mathbb{R}^{\prime}$, $x<'y$ and $y<'z$, i. e., $0 <_{\mathbb{R}^{\prime}}y-x \in  \mathbb{R}$ and $0 <_{\mathbb{R}^{\prime}}z-y \in \mathbb{R}$, then $0 <_{\mathbb{R}^{\prime}}z-x=(y-x)+(z-y) \in \mathbb{R}$ (for $\mathbb{R}$ is certainly closed under $+$), which means that $x <' z$. So $<'$ is transitive. Now if $0<'x$ and $0<'y$, i.e., $0 <_{\mathbb{R}^{\prime}}x-0=x \in \mathbb{R}$ and $0 <_{\mathbb{R}^{\prime}}y-0=y \in \mathbb{R}$ then surely $0 <_{\mathbb{R}^{\prime}}xy-0 = xy \in \mathbb{R}$ (for $\mathbb{R}$ is also closed under $\times$), i.e., $0<'xy$. Finally, if $x<'y$,  i. e., $0 <_{\mathbb{R}^{\prime}}y-x \in  \mathbb{R}$ then since $z-z=0$, it follows that  $0 <_{\mathbb{R}^{\prime}} (y+z) - (x+z) = (y-x) + (z-z)  \in  \mathbb{R}$, which means that $x+z<'y+z$.

\bigskip
Now consider a Routley-Meyer model $N^{\prime}=\langle W,D,R,^*,s,V\rangle$ where

\begin{itemize}
\item $W=\{s, t\}$
\item $D=\mathbb{R}^{\prime}$
\item All constants and function symbols of the language of RCF are interpreted on $\mathbb{R}^{\prime}$ as in  $\langle \mathbb{R}^{\prime}, <_{\mathbb{R}^{\prime}}, +_{\mathbb{R}^{\prime}}, -_{\mathbb{R}^{\prime}},  ^{-1_\mathbb{R}^{\prime}}, \times_{\mathbb{R}^{\prime}}, 0_{\mathbb{R}^{\prime}}, 1_{\mathbb{R}^{\prime}} \rangle$.
\item $R=\{\langle s,s,s\rangle,\langle s,t,t\rangle,\langle t,s,s\rangle,\langle t,s,t\rangle,\langle t,t,s\rangle,\langle t,t,t\rangle, \langle s,t,s\rangle\}$
\item $V$ assigns the symbol $<$ the  (strict) ordering $<_{  \mathbb{R}^{\prime}}$ on $\mathbb{R}^{\prime}$ at world $s$. At world $t$, $V$ assigns the ordering $<'$ as the interpretation of the symbol $<$. Note that since $<' \subset <_{\mathbb{R}^{\prime}}$, $V$ is an admissible valuation in the Routley-Meyer semantics.  At every world, $=$ is interpreted as real identity.
\item $s$ is the designated world.
\item $^*=\{\langle s,t\rangle,\langle t,s\rangle\}$.
\end{itemize}

This structure is a model of RCF. Next take any $r \in \{x \in \mathbb{R}^{\prime} : 0_{\mathbb{R}^{\prime}} <_{\mathbb{R}^{\prime}} x\}\setminus \mathbb{R}$ (recall that we made sure that such an $r$ exists). Now, $(r+_{\mathbb{R}^{\prime}} 1_{\mathbb{R}^{\prime}}) - 1_{\mathbb{R}^{\prime}} = r$. But then it cannot be that $1_{\mathbb{R}^{\prime}} <^{\prime}(r+_{\mathbb{R}^{\prime}} 1_{\mathbb{R}^{\prime}})$. Hence,
\begin{center}
 $N^{\prime}, s \nvDash \forall y, z ((0 < y < z \land \neg(y<z)) \rightarrow \neg \forall x (x=x) )$.
\end{center}
Observe that in this model  $\neg \forall x (x=x) )$ is essentially just $\bot$. 

Now take the model $N^{\prime \prime}$ of RCF which is just $N^{\prime}$ but with $D = \mathbb{R}$. With this modification, the ordering $<^{\prime}$ basically collapses to $<_{\mathbb{R}}$. It is easy to see then that 
\begin{center}
 $N^{\prime \prime}, s \vDash \forall y, z ((0 < y < z \land \neg(y<z)) \rightarrow \neg \forall x (x=x) )$.
\end{center}
Finally, one can show by induction on formula complexity that if $A(x)$ is a quantifier free relevant formula in the full language of RCF and $\overline{a}$ a sequence of objects from the domain $\mathbb{R}$ of  $N^{\prime \prime}$, $w \in W$, then
\begin{center}
$N^{\prime \prime}, w \vDash A[\overline{a}]$ iff $N^{\prime }, w \vDash A[\overline{a}]$.
\end{center}
 But with the above results this implies that the formula  in the free variable $x$, 
\begin{center}
$\forall y, z ((x < y < z \land \neg(y<z)) \rightarrow \neg \forall x (x=x) )$
\end{center}
 cannot be equivalent to any quantifierless formula of the language of RCF\footnote{In fact, the presence of $\rightarrow$ is not essential here, the formula $\exists y, z ((x < y)\land(y < z) \land \neg(y<z))$ would also do the trick.}, so the theory lacks QE.

\end{proof}

\begin{thm} Let \textbf{L} be any relevant logic between \textbf{B} and \textbf{RM}. Then QE fails for Presburger arithmetic in \textbf{L}.
\end{thm}

\begin{proof}
This follows as for RCF except that we work with  $\mathbb{Z}$ rather than $\mathbb{R}$, and a smaller signature. The reader can fill in the details.
\end{proof}

\begin{thm} Let \textbf{L} be any relevant logic between \textbf{B} and \textbf{RM}. Then QE fails for DOAG in \textbf{L}.
\end{thm}

\begin{proof}
The proof is left to the reader and it follows as for the case of RCF again, by using $\mathbb{Q}$ rather than $\mathbb{R}$.
\end{proof}

To end this section we wish to remark that the failure of QE presented here cannot be fixed by taking a weaker notion of QE where we merely demand interdeducibility with a quantifier free formula as opposed to equivalence in the sense of a relevant biconditional.  The reason is that our arguments show that the formulas witnessing the failure of QE cannot be interdeducible with any quantifierless formula either. To see this note that, for instance, in the proof of Theorem \ref{t}, we actually showed that there cannot be any quantifierless formula $A(x)$ such that:
 
\begin{center}
$N^{\prime \prime}, s \vDash \forall y, z ((0 < y < z \land \neg(y<z)) \rightarrow \neg \forall x (x=x) )$ iff $N^{\prime }, s \vDash A[0]$,
\end{center}
which means that no such quantifierless formula  can be interdeducible with $\forall y, z ((x < y < z \land \neg(y<z)) \rightarrow \neg \forall x (x=x) )$.

\section{Theories with Quantifier Elimination}\label{sec:qe}

 In this section we study under which circumstances can we get QE back. Some proofs  are simply arguments from \cite{Marker02}, given in the Routley-Meyer semantic framework.

\begin{pro}\label{2}
Let \textbf{L} be any logic  extending \textbf{B} + \textsf{K} such that its frames satisfy the condition $\forall x \exists y, z (Rxyz)$, $M=\langle W,R,D,^*, s, V\rangle$  an \textbf{L}-model for DLO, $A(\overline{x})$ a formula of DLO, and $\overline{a}$ a sequence of elements of D. Then for any $w \in W$, $M, w \vDash A [\overline{a}]$ iff  $M, s \vDash A [\overline{a}]$. 
\end{pro}
\begin{proof}
The frame condition for \textsf{K} (the principle $A\rightarrow(B\rightarrow A))$ is:
\begin{center}
(1) \,\,\,\,\,\, $R\alpha \beta \gamma$ only if $Rs\alpha \gamma$.
\end{center}
Now we proceed by induction of the complexity of $A(\overline{x})$. Suppose first that $A(\overline{x})$ is an atomic formula of the form $x_i < y_j$. Then if $M, s \vDash A [\overline{a}]$, $M, w \vDash A [\overline{a}]$  follows by the Hereditary condition and the fact that $Rssw$ in the presence of (1) since $Rsww$. For the converse suppose that $M, w \vDash A [\overline{a}]$, which means that $M, w \vDash x_i < y_j [a_ia_j]$. Now if $M, s \nvDash x_i < y_j [a_ia_j]$, since $M$ is a model for $DLO$, either $M, s \vDash x_i = y_j [a_ia_j]$ or $M, s \vDash y_j <   x_i  [a_ia_j]$. If the first, it must be that $M, w \vDash x_i = y_j [a_ia_j]$ and since $(x_i = y_j \land x_i < y_j)  \rightarrow x_i < x_i $ is a theorem of \textbf{L}, it follows that $M, s \vDash (x_i = y_j \land x_i < y_j)  \rightarrow x_i < x_i   [a_ia_j]$, so 
 since $Rsww$, it must be that  $M, w \vDash  x_i < x_i   [a_ia_j]$, which is impossible. If, on the other hand, $M, s \vDash y_j <   x_i  [a_ia_j]$, then $M, w \vDash y_j <   x_i  [a_ia_j]$ and given that $Rsww$ it must be the case that  $M, w \vDash x_i <   x_i  [a_ia_j]$, i.e., $M, w \vDash x_i <   x_i  [a_i]$, which is impossible again. Hence, $M, s \vDash x_i < y_j [a_ia_j]$, as desired.

Suppose next that $A(\overline{x})$ is of the form $\neg B(\overline{x})$. If $M, s \vDash \neg B [\overline{a}]$ then $M, w \vDash \neg B [\overline{a}]$ by the Hereditary condition. Now if $M, s \nvDash \neg B [\overline{a}]$ then $M, s^* \vDash  B [\overline{a}]$ and, by a double application of the inductive hypothesis, $M, s \vDash  B [\overline{a}]$ and $M, w^* \vDash  B [\overline{a}]$, which in turn implies that $M, w\nvDash  \neg B [\overline{a}]$, as wanted.

Let $A(\overline{x})$ be of the form $B \rightarrow C$. One direction of the result follows by the Hereditary condition again. For the converse, suppose that $M, s \nvDash B \rightarrow C [\overline{a}]$, so there are $y, z \in W$ such that $M, y \vDash B  [\overline{a}]$ and $M, z \nvDash C  [\overline{a}]$. By inductive hypothesis, it follows that $M, s \vDash B  [\overline{a}]$ and $M, s \nvDash C  [\overline{a}]$. By assumption there are $u, v \in W$ such that $Rwuv$ (this is the only point in the proof where we need this assumption). Consequently, by inductive hypothesis again that $M, u \vDash B  [\overline{a}]$ and $M, v \nvDash C  [\overline{a}]$. But then $M, w \nvDash B \rightarrow C [\overline{a}]$, as desired.

The remaining cases are similarly straightforward.
\end{proof}

\begin{pro}\label{c} \emph{(Cantor's Theorem)}
Let \textbf{L} be any logic  extending \textbf{B} + \textsf{K} such that its frames satisfy the condition $\forall x \exists y, z (Rxyz)$, $M=\langle W,R,D,^*, s, V\rangle, N= \langle W^{\prime},R^{\prime},D^{\prime},^{*^{\prime}}, s^{\prime}, V^{\prime}\rangle$  two countable \textbf{L}-models for DLO, and $a_1, \dots , a_n \in D$, $b_1, \dots , b_n \in D^{\prime}$ sequences such that $a_1 < \dots < a_n$ and $b_1 < \dots < b_n$ hold at $s$ and $s^{\prime}$ respectively. Then there is a bijective mapping $f: D \longrightarrow D^{\prime}$ such that $f(a_i)=b_i$ and for any formula $A(\overline{x})$ of DLO and  sequence of elements $\overline{a}$ of D,
\begin{center}
 $M, s \vDash A [\overline{a}]$ iff  $N, s^{\prime} \vDash A [f(\overline{a})]$. 
\end{center}
\end{pro}

\begin{proof}
The function $f$ is built by a back-and-forth  argument.\footnote{See \cite{Marker02} for examples of such arguments in a classical setting.} We make sure by construction that  $M, s \vDash x<y [\overline{a}]$ iff  $N, s^{\prime} \vDash x< y [f(\overline{a})]$.

Let $c_0, c_1, c_2 \dots $ and $d_0, d_1, d_2 \dots $  be enumerations of $D\setminus \{a_1, \dots,  a_n\}$ and $D^{\prime}\setminus \{b_1, \dots,  b_n\}$ respectively. Next we build a sequence $f_0 \subseteq f_1 \subseteq \dots$ of bijections $f_i: D_i \longrightarrow D^{\prime}_i $ such that $ D_i \subset D$ and $ D^{\prime}_i \subset D^{\prime}$ are finite and $M, s \vDash x<y [a, b]$ iff  $N, s^{\prime} \vDash x< y [f_i(a), f_i(b)]$  for any $a, b \in D_i$. The idea is to build the sequences such that $\bigcup D_i =D$ and $\bigcup D^{\prime}_i =D^{\prime}$, so $\bigcup f_i$ would be the $f$ required by the theorem.

{\sc Stage} 0: Just put $D_0 = \{a_1, \dots,  a_n\}$, $D^{\prime}_0 = \{b_1, \dots,  b_n\}$ and let $f_0$ be the mapping $a_i \mapsto b_i$.

{\sc Stage} n+1=2m+1: When  $c_m \in D_n$ then just put $D_{n+1}=D_n$, $D^{\prime}_{n+1}$ and $f_{n+1}=f_n$.
On the other hand suppose that $c_m \notin D_n$. We want $D_{n+1}$ to be $D_n \cup \{c_m\}$ in this case. This means that we need to choose $d \in D^{\prime}\setminus D^{\prime}_{n}$ carefully for $f_{n+1}=f_n \cup \{\langle c_m, d\rangle\}$ and  $D^{\prime}_{n+1}= D^{\prime}_{n}\cup\{d\}$ be as required by the construction.

Keep in mind that we are working with models of DLO. Hence, there are three mutually exclusive possibilities:
\begin{itemize}
\item[(1)] $M, s \vDash x<y [c, c_m]$ for each $c \in D_n$
\item[(2)] $M, s \vDash x<y  [c_m, c]$ for each $c \in D_n$
\item[(3)] there are $a, b \in  D_n$ such that $M, s \vDash x<y  [a, b]$,   $M, s \vDash x<y < z  [a, c_m, b]$ and for all $c\in D_n$ either  $M, s \vDash x<y \vee x=y  [c, a]$ or $M, s \vDash x<y \vee x=y  [b, c]$.
\end{itemize}

If (1) holds, then since  $D_n^{\prime}$ is finite, we can find $d  \in D^{\prime}\setminus D^{\prime}_{n}$ such that $N, s^{\prime} \vDash x<y [c, d]$ for each $c \in  D^{\prime}_{n}$ (this comes essentially from the same argument showing that $DLO$ has no models with a finite domain). A similar thing is the case when (2) holds. In case (3), given that $N, s^{\prime} \vDash x<y  [f_n(a), f_n(b)]$ by construction, since $N$ is a model of $DLO$ just take $d  \in D^{\prime}\setminus D^{\prime}_{n}$ such that $N, s^{\prime} \vDash x< z<y  [f_n(a), d, f_n(b)]$.

{\sc Stage} n+1=2m+2: When $d_m \in D^{\prime}_n$, we simply let $D_{n+1}=D_n$, $D^{\prime}_{n+1}= D^{\prime}_{n}$ and $f_{n+1}=f_n$. Otherwise, we proceed in an analogous way to the previous case making sure that $D_{n+1}$, $D^{\prime}_{n+1}$ and $f_{n+1}$ are such that $d_m \in D^{\prime}_{n+1}$.

After having constructed $f$, we establish the equivalence in the proposition by induction on the complexity of $A(\overline{x})$. If $A(\overline{x})$ is atomic, then the result follows by construction of $f$.

Suppose next that $A(\overline{x})$ is of the form $\neg B(\overline{x})$. If $M, s \vDash \neg B [\overline{a}]$, so $M, s^* \nvDash  B [\overline{a}]$, hence $M, s \nvDash  B [\overline{a}]$ by Proposition \ref{2}. So $N, s^{\prime} \nvDash  B [f(\overline{a})]$ by inductive hypothesis, and $N, s^{\prime *^{\prime}} \nvDash  B [f(\overline{a})]$ by the previous theorem again, so $N, s^{\prime *^{\prime}} \vDash \neg  B [f(\overline{a})]$ as desired. The converse is symmetric.

Now let $A(\overline{x})$ be of the form $B \rightarrow C$. Suppose that $M, s \vDash B \rightarrow C [\overline{a}]$. Assume further that $N, s^{\prime} \vDash B [f(\overline{a})]$, so by inductive hypothesis, it must be that $M, s  \vDash B [\overline{a}]$. Since $Rsss$, $M, s  \vDash C [\overline{a}]$, which, by inductive hypothesis, means that  $N, s^{\prime} \vDash C [f(\overline{a})]$. Then, by Proposition \ref{2}, it must be that $N, w \vDash B [f(\overline{a})]$ only if $N, w \vDash C [f(\overline{a})]$ for any $w \in W^{\prime}$, so, in particular,  $N, s^{\prime} \vDash B \rightarrow C [f(\overline{a})]$ holds. The converse follows by symmetry.

The remaining cases are straightforward and the quantificational case uses the fact that $f$ is surjective.
\end{proof}

\begin{pro} \label{3} Let \textbf{L} be any logic  extending \textbf{B} + \textsf{K} such that its frames satisfy the condition $\forall x \exists y, z (Rxyz)$. Then DLO in  \textbf{L} is complete.

\end{pro}

\begin{proof} Suppose not, that is, there are two \textbf{L}-models $M=\langle W,R,D,^*, s, V\rangle$ and $ N= \langle W^{\prime},R^{\prime},D^{\prime},^{*^{\prime}}, s^{\prime}, V^{\prime}\rangle$ for $DLO$ such that there is a sentence $A$ of $DLO$ for which $M, s \vDash A$ but $N, s^{\prime} \nvDash A$. The cardinality of both $D$ and $D^{\prime}$ has to be infinite. Since the language of quantificational relevant logic can be embedded in a satisfaction preserving way into the language of classical first order logic, by the  classical downward L\"owenheim-Skolem theorem, $M$ and $N$ can both be assumed to be countable. By Proposition \ref{c}, they have to make exactly the same sentences hold at  $s$ and $s^{\prime}$, which contradicts our assumption. \end{proof}

\begin{pro}  Let \textbf{L} be any logic  extending \textbf{B} + \textsf{K} such that its frames satisfy the condition $\forall x \exists y, z (Rxyz)$ and \textbf{L} is strongly complete with respect to the Routley-Meyer semantics. Then DLO in  \textbf{L} is negation complete.

\end{pro}

\begin{proof} Take any sentence $A$ of $DLO$. Consider an arbitrary \textbf{L}-model $M$ for $DLO$. Then either $M, s \vDash A$ or  $M, s \nvDash A$. The second implies $M, s^* \nvDash A$ by Proposition \ref{2}, with $s^*$ substituted for $w$. This is equivalent to $M, s \vDash \neg A$. Since, according to Proposition  \ref{3}, all \textbf{L}-models of $DLO$ make exactly the same sentences true, it must be the case that either every \textbf{L}-model for $DLO$ make $A$ the case or that every \textbf{L}-model for $DLO$ make $\neg A$ hold. Hence, either $DLO \vDash A$ or $DLO \vDash \neg A$.

\end{proof}

\begin{thm}   Let \textbf{L} be any logic  extending \textbf{B} + \textsf{K} such that its frames satisfy the condition $\forall x \exists y, z (Rxyz)$ and \textbf{L} is strongly complete with respect to the Routley-Meyer semantics. Then DLO in  \textbf{L} has QE.

\end{thm}

\begin{proof} Take any countable model $M=\langle W,R,D,^*, s, V\rangle$  of $DLO$ in  \textbf{L}. First, if $A$ is a sentence and $M, s \vDash A$, using Proposition \ref {2}, it must be that $M, s \vDash A \leftrightarrow x =x$, which means that $DLO \vDash  A \leftrightarrow x =x$ by the completeness of $DLO$. On the other hand if $M, s \nvDash A$, we similarly obtain that $DLO \vDash  A \leftrightarrow x < x$.

Suppose now that $A(\overline{x})$ is a formula in the free variables $x_1, \dots, x_n$. Let $h: \{\langle i, j\rangle : 1 \leq  i < j \leq n\} \longrightarrow \{0, 1, 2\}$ and consider the formula $B_h(x_1, \dots, x_n)$ defined as follows:

\begin{center}
$ \bigwedge_{h(i, j)=0} x_i=x_j \land   \bigwedge_{h(i, j)=1} x_i<x_j  \land  \bigwedge_{h(i, j)=2} x_j<x_i$.
\end{center}
Write $\Lambda_A$ for the set of all $B_h(x_1, \dots, x_n)$ with $h: \{\langle i, j\rangle : 1 \leq  i < j \leq n\} \longrightarrow \{0, 1, 2\}$ such that there is a sequence of elements $\overline{a}$ of $D$ for which $M, 0 \vDash A(x_1, \dots, x_n) \land B_h(x_1, \dots, x_n) [\overline{a}]$. $\Lambda_A$ will, of course, always be finite.

If $\Lambda_A = \emptyset$, then for all sequences $\overline{a}$ of $D$ it must be that $M, s \nvDash A[\overline{a}]$ for note that every   $\overline{a}$ would satisfy some $B_h(x_1, \dots, x_n)$. This means that $M, s \vDash \forall x_1, \dots, x_n (A \leftrightarrow x_{n+1} < x_{n+1})$, which implies that $DLO \vDash \forall x_1, \dots, x_n (A \leftrightarrow x_{n+1} < x_{n+1})$.

On the other hand, if $\Lambda_A \neq \emptyset$, we claim that  $M, s \vDash \forall x_1, \dots, x_n (A \leftrightarrow \bigvee  \Lambda_A)$, which implies that  $DLO \vDash \forall x_1, \dots, x_n (A \leftrightarrow \bigvee  \Lambda_A)$. The half  $M, s \vDash \forall x_1, \dots, x_n (A \rightarrow \bigvee  \Lambda_A)$ is clear since, as we said, every sequence of objects of $D$ satisfies some $B_h(x_1, \dots, x_n)$. For the other half, let $a_1, \dots a_n \in D$ and $M, s \vDash B_h[a_1, \dots a_n]$ for some $B_h(x_1, \dots, x_n) \in \Lambda_A$. By definition of $\Lambda_A$ there is a sequence of objects $b_1, \dots b_n \in D$ such that $M, s \vDash A(x_1, \dots, x_n) \land B_h(x_1, \dots, x_n) [b_1, \dots b_n ]$. Now take the mapping $f: D \longrightarrow D$ such that $f(a_i)=b_i$ given by our version of Cantor's theorem. Since  $M, s \vDash A [b_1, \dots b_n ]$, it must be that $M, s \vDash A [a_1, \dots a_n ]$. This shows that $M, s \vDash \forall x_1, \dots, x_n ( \bigvee  \Lambda_A \rightarrow A)$, as desired. 

\end{proof}

An example of a logic in which $DLO$ will have all the nice features described above is  \textbf{B} + \textsf{C} +\textsf{K} (since all frames for this logic satisfy $\forall x(Rxsx)$).

\begin{pro}   Let \textbf{L} be any logic  extending \textbf{B} + \textsf{K}. Then for any model $M=\langle W,R,D,^*, s, V\rangle$ of $DLO$ in \textbf{L} and formula $A(\overline{x})$  there is a quantifier free formula $B(\overline{x})$ such that for every sequence of objects $\overline{a}$ in $D$, $M, s\vDash  A[\overline{a}]$ iff $ M, s \vDash B[\overline{a}]$.
\end{pro}
\begin{proof} Using the downward L\"owenheim-Skolem theorem and the fact that Routley-Meyer models are first order structures, it is not difficult to see that it suffices to establish the theorem for models with a  countable domain of objects.
By examining the proof of the previous theorem, it is easy to see how it can be adapted \emph{mutatis mutandis} up to the case that $\Lambda_A \neq \emptyset$. At this point what is required is an instance of Cantor's theorem saying that where $M=\langle W,R,D,^*, s, V\rangle$ is a countable \textbf{L}-model for DLO, and $a_1, \dots , a_n \in D$, $b_1, \dots , b_n \in D$ sequences such that $a_1 < \dots < a_n$ and $b_1 < \dots < b_n$ hold at $s$, then there is a bijective mapping $f: D \longrightarrow D$ such that $f(a_i)=b_i$ and for any formula $A(\overline{x})$ of DLO, $w \in W$ and  sequence of elements $\overline{a}$ of D,
\begin{center}
 $M, w \vDash A [\overline{a}]$ iff  $M, w \vDash A [f(\overline{a})]$. 
\end{center}
A simple modification of the proof of Cantor's theorem suffices to establish it. The construction of $f$ is basically as before.  To prove the above equivalence  we need to note first that for any sequence of objects $\overline{a}$, $M, w \vDash A [\overline{a}]$ iff $M, s \vDash A [\overline{a}]$ if $A$ is atomic. This follows by the same argument used for the basis of the inductive proof of Proposition \ref{2}. Now a simple induction on formula complexity takes care of the rest.

\end{proof}

So if the interest in QE was to reduce the complexity of the definable sets of any given model of $DLO$, \textbf{B} + \textsf{K} is a logic extending \textbf{B} which has this effect, and we see no way to regain this result without employing \textsf{K}, or some similar principle which has the effect of eliminating all the worlds in the Routley-Meyer model except the designated world. It is important to note that this theorem implies that one cannot use the same strategy we used to refute QE for $DLO$ in  \textbf{B} in order to refute QE for  $DLO$ in  \textbf{B} + \textsf{K}. For there is no model $M$ for $DLO$ in  \textbf{B} + \textsf{K} and formula $A(\overline{x})$ such that $A(\overline{x})$ does not define a set already definable by a quantifierless   formula at $s$ in $M$.

\begin{thm}   Let \textbf{L} be any logic  extending \textbf{B} such that its frames satisfy the condition $\forall x \exists y, z (Rxyz)$. Then ACF in  \textbf{L} has QE.
\end{thm}

\begin{proof} We start by showing that if $M=\langle W,R,D,^*, s, V\rangle$  is an \textbf{L}-model for ACF, $A(\overline{x})$ a formula of ACF, and $\overline{a}$ a sequence of elements of D. Then for any $w \in W$, $M, w \vDash A [\overline{a}]$ iff  $M, s \vDash A [\overline{a}]$. This follows by an easy induction on formula complexity as before. So, in these models, a formula of the form $A \rightarrow B$ comes down semantically to the same as a formula of the form $\neg A \vee B$, which means that all relevant formulas in these models are equivalent to $\rightarrow$-free formulas. Since ACF based on classical logic has QE, it suffices to show that every theorem of classical ACF is also a theorem of ACF  in  \textbf{L}.   But the contrapositive of this is easily established, for if  $M=\langle W,R,D,^*, s, V\rangle$  is an \textbf{L}-model for ACF and  $M, s \nvDash A $ for some relevant formula $A$ (which without loss of generality can be taken to be $\rightarrow$-free), then a classical model refuting $A$ is easily extracted by interpreting function symbols in $D$ as they are interpreted at world $s$ in the Routley-Meyer model.\footnote{This kind of argument would have also established QE for DLO for logics extending \textbf{B} + \textsf{K} and with frames satisfying seriality. However, our previous argument for DLO was much more informative.}  \end{proof}

The above proof works, in fact, for any theory whose classical counterpart has QE and that does not contain non-logical symbols other than function and constants.

This positive result (as the one for DLO) is, of course, sensitive to a change of our background assumptions on equality. In this paper we have taken the predicate to mean real identity under all circumstances. This position is certainly rejected by most relevant logicians, though. This is because it produces validities like $A \rightarrow \forall x (x=x)$ for any arbitrary $A$.  Relevant logicians would normally require just something like principles (R), (T) and (L) in the theorem below to hold. In such case, QE can be shown to fail by arguments similar to the ones in \S \ref{sec:fqe}. This indicates that there is very little hope for  positive results on QE when 
$=$ is allowed all the freedom most relevant logicians would grant it. In other words, for the dedicated relevant logician, QE fails for quite simple reasons in most circumstances.\footnote{Note, our results do not directly rule out building theories with QE by means of Skolemization, so we cannot make this claim quite unequivocally.}

\begin{thm} Suppose that all we require semantically from the predicate $=$ is that it satisfies the principles of reflexivity, transitivity, and a version of Leibniz's law
\begin{itemize}
\item[(R)] $\forall x x=x$,
\item[(T)]$\forall x, y, z (x=y \wedge y=z \rightarrow x=z)$,
\item[(L)]$\forall x, y (x=y \wedge A(x) \rightarrow A(y))$ for all relevant formulas $A$.
\end{itemize}
  
  Let \textbf{L} be any logic between \textbf{B} and \textbf{RM}. Then   QE fails for ACF in  \textbf{L}.
\end{thm}

\begin{proof} We argue in a rather similar way as in the failure of QE for RCF. However, this time we use the classical structure  $\langle \mathbb{A},  +_{\mathbb{A}}, -_{\mathbb{A}},  ^{-1_\mathbb{A}}, \times_{\mathbb{A}}, 0_{\mathbb{A}}, 1_{\mathbb{A}} \rangle$ of algebraic numbers and some uncountable elementary extension (again in the classical sense)
$\langle \mathbb{A}^{\prime}, +_{\mathbb{A}^{\prime}}, -_{\mathbb{A}^{\prime}},  ^{-1_\mathbb{A}^{\prime}}, \times_{\mathbb{A}^{\prime}}, 0_{\mathbb{A}^{\prime}}, 1_{\mathbb{A}^{\prime}} \rangle$.

Now consider a Routley-Meyer model $N^{\prime}=\langle W,D,R,^*,s,V\rangle$ where

\begin{itemize}
\item $W=\{s, t\}$
\item $D=\mathbb{A}^{\prime}$
\item All constants and function symbols of the language of ACF are interpreted on $\mathbb{A}^{\prime}$ as in  $\langle \mathbb{A}^{\prime},  +_{\mathbb{A}^{\prime}}, -_{\mathbb{A}^{\prime}},  ^{-1_\mathbb{A}^{\prime}}, \times_{\mathbb{A}^{\prime}}, 0_{\mathbb{A}^{\prime}}, 1_{\mathbb{A}^{\prime}} \rangle$.
\item $R=\{\langle s,s,s\rangle,\langle s,t,t\rangle,\langle t,s,s\rangle,\langle t,s,t\rangle,\langle t,t,s\rangle,\langle t,t,t\rangle, \langle s,t,s\rangle\}$
\item $V$ assigns to the symbol $=$ real equality on $\mathbb{A}^{\prime}$ at world $s$. At world $t$, $V$ assigns equality restricted to $\mathbb{A}$ as the interpretation of the symbol $=$. Note that since $V(=, t) \subset V(=, s)$, $V$ is a valuation in the Routley-Meyer semantics.  
\item $s$ is the designated world.
\item $^*=\{\langle s,t\rangle,\langle t,s\rangle\}$.
\end{itemize}

This structure is a model of  ACF. Next take any $r \in  \mathbb{A}^{\prime}\setminus \mathbb{A}$ (recall that we made sure that such an $r$ exists since $\mathbb{A}^{\prime}$ is uncountable while $\mathbb{A}$ is countable). Now, 
\begin{center}
 $N^{\prime}, s \vDash \neg(r=0) \wedge \neg (r=r)$,
\end{center}
so, 
\begin{center}
 $N^{\prime}, s \vDash \exists x( \neg(x=0) \wedge \neg (x=x))$.
\end{center}

Now take the model $N^{\prime \prime}$ of ACF which is just $N^{\prime}$ but with $D = \mathbb{A}$. With this modification, the predicate $=$ becomes real equality at every world in the model, so 
\begin{center}
  $N^{\prime \prime}, s \nvDash \exists x( \neg(x=0) \wedge \neg (x=x))$.
\end{center}
One can show by induction on formula complexity that if $A(x)$ is a quantifier free relevant formula in the language of ACF and $\overline{a}$ a sequence of objects from the domain $\mathbb{A}$ of  $N^{\prime \prime}$, $w \in W$, then
\begin{center}
$N^{\prime \prime}, w \vDash A[\overline{a}]$ iff $N^{\prime }, w \vDash A[\overline{a}]$.
\end{center}
 But with the above results this implies that the formula  in the free variable $y$, 
\begin{center}
$\exists x( \neg(x=y) \wedge \neg (x=x))$
\end{center}
 cannot be equivalent to any quantifierless formula of the language of ACF, so the theory does not have QE.
 
\end{proof}

\section{Conclusion}

One of the main morals is that the property of quantifier elimination for a number of well-known mathematical theories (involving non-logical \emph{predicate} -- and not just \emph{function} -- symbols) is easily lost in many relevant logics. Naturally, this is because logics that bring the property back are nearly classical, in that they enforce strong extensional equivalences which are rejected in a relevant setting. So when working in a more fine-tuned context, with fewer principles available, a property like QE   which marks a theory's simplicity should no longer be expected to hold. As anyone who has tried to do non-classical mathematics can testify,  there is a substantial jump in complexity from classical to non-classical mathematics.

It is of definite importance to the study of relevant logic and mathematics that \textsf{K} plays a key role in capturing classical results for DLO. Thinning does play a key role, and if the relevance properties of the logic involve rejecting \textsf{K}, then logics which respect relevance provide a view of standard mathematical theories far different from the orthodoxy. This should not be particularly surprising in itself, but what is surprising is that in a logic as weak as \textbf{B}, the addition of just \textsf{K} and a plausible seriality condition on Routley-Meyer frames goes a long way to capturing classical results. That is, insofar as one may want QE and categoricity considerations are useful for purposes (1)--(4) in \S 1, at least \textsf{K} is required (and indeed, in addition to \textsf{K} nothing beyond \textbf{B} is required). The dedicated relevantist, classical logician, or non-partisan observer can, of course, take these results negatively, as showing that irrelevance really plays a core role in standard mathematical model-theoretic reasoning. However, there is a more interesting response available, to which we have gestured. This is that even in the case of simple mathematical theories like those considered here, a logical approach maintaining relevance opens a view which, as Meyer and Mortensen claimed, ``cannot impoverish insight into the nature of mathematical structures, but rather can only enrich it'' \cite{MeyerMortensen84}.

\label{sec:con}

\section*{Acknowledgements}

There are many people deserving of thanks for input on various stages of this project. We would like to thank the audiences of the \emph{Frontiers of Non-Classicality: Logic, Mathematics, Philosophy} workshop held in Auckland in Jan. of 2016 and the \emph{Third Workshop} held in Edmonton in May of 2016 for their comments, and the organisers of these events. More particular thanks is due to Zach Weber and an anonymous referee for valuable commentary on drafts of this paper.  

Tedder: I would like to thank the Department of Philosophy at the University of Connecticut for travel funding for the Auckland and Edmonton workshops.

Badia: I acknowledge the support by the Austrian Science Fund
(FWF): project I 1923-N25 (\emph{New perspectives on residuated posets}).

\bibliography{BDLO_QE}
\bibliographystyle{amsplain}
\nocite{*}

\end{document}